\documentclass[12pt]{amsart}
\usepackage{fancybox}
\usepackage{ascmac}
\usepackage{delarray}
\usepackage{enumerate}
\usepackage{amsmath,amssymb}
\usepackage{color}
\usepackage{comment}

\usepackage{mathrsfs}
\usepackage{paralist}

\makeatletter
 \@addtoreset{equation}{subsection}
\makeatother

\def\re{\mathbb{R}}

\def\N{\mathbb{N}}

\def\pd{\partial}

\def\la{\lambda}
\def\al{\alpha}

\def\({\left(}
\def\){\right)}
\def\pd{\partial}

\def\BOX{{\setlength{\unitlength}{1pt}\begin{picture}(8,8)
	\put(1,1){\framebox(6,6)}\end{picture}}\ }
\def\qed{\hfill\BOX\vskip1em\par}

\def\intO{\int_{\Omega}}

\numberwithin{equation}{section}
\newtheorem{theorem}{Theorem}[section]

\newtheorem{lemma}[theorem]{Lemma}

\newtheorem{remark}[theorem]{Remark}

\numberwithin{theorem}{section}

\begin{document}
\title[The equality case]{Another approach to the equality case of the sharp logarithmic Sobolev inequality}
\author[F. Feo]{Filomena Feo$^1$}
\author[F. Takahashi]{Futoshi Takahashi$^2$}

\begin{abstract}
In this note, we characterize the equality case of the sharp $L^2$-Euclidean logarithmic Sobolev inequality with monomial weights,
exploiting the idea by Bobkov and Ledoux \cite{Bob}.
Our approach is new even in the unweighted case.
Also, we show that the same strategy yields the equality case of the sharp $L^p$-Euclidean logarithmic Sobolev inequality for $1 < p < \infty$ with arbitrary norm on $\re^n$, 
if the inequality is unweighted. 
\end{abstract}

\setcounter{footnote}{1}
\footnotetext{
Dipartimento di Ingegneria, Universit\`{a} degli Studi di
Napoli \textquotedblleft Parthenope\textquotedblright, Centro
Direzionale Isola C4 80143 Napoli, Italy. \\
 e-mail:{\tt filomena.feo@uniparthenope.it}}

\setcounter{footnote}{2}
\footnotetext{
Department of Mathematics, Osaka Metropolitan University,
3-3-138 Sugimoto, Sumiyoshi-ku, Osaka 558-8585, Japan. \\
e-mail:{\tt futoshi@omu.ac.jp}}

\medskip

\maketitle

\bigskip

\section{Introduction}

Let $1 \le p < \infty$.
The Euclidean $L^p$-logarithmic Sobolev inequality states that
\begin{equation}
\label{ELSp}
	\int_{\re^n} |f|^p \log |f|^p dx \le \frac{n}{p} \log \( \mathcal{L}_p \int_{\re^n} |\nabla f|^p dx \)
\end{equation}
holds for any function $f \in W^{1,p}(\re^n)$ such that $\int_{\re^n} |f|^p dx = 1$. 
This form of inequality was first proved by Weissler \cite{Weissler} for $p = 2$,  
Ledoux \cite{Ledoux} for $p=1$ (but in this case, $BV(\re^n)$ is an appropriate function space rather than the Sobolev space $W^{1,1}(\re^n)$),
del Pino and Dolbeault \cite{DD} for $1 < p < n$,
and Gentil \cite{Gentil} (see also \cite{Fujita}) for $1 < p < \infty$.
Actually, Gentil extends the result in \cite{DD} not only for all $p > 1$, but also for any norm on $\re^n$ other than usual Euclidean norm.
Here the sharp constant $\mathcal{L}_p$ is given by
\begin{align*}
	&\mathcal{L}_1 = \frac{1}{n} \pi^{-1/2} \(\Gamma(n/2+1)\)^{1/n}, \quad p = 1, \\
	&\mathcal{L}_{p} = \frac{p}{n} \(\frac{p-1}{e}\)^{p-1} \pi^{-p/2} \(\frac{\Gamma(n/2+1)}{\Gamma(n/p' + 1)} \)^{p/n}, \quad p > 1,
\end{align*}
here and throughout of the paper, $p' = \frac{p}{p-1}$ for $p > 1$.
For $p=1$, there is no function in $W^{1,1}(\re^n)$ which achieves the equality. 
In fact, Beckner \cite{Beckner3} proved that the equality in \eqref{ELSp} holds if and only if $f \in BV(\re^n)$ is the characteristic function of a ball.
For $1 < p < n$, it is proved in \cite{DD} that the equality in \eqref{ELSp} holds if and only if $f$ has the form
\[
	f(x) = \( \pi^{n/2} \(\frac{\sigma}{p}\)^{n/p'} \frac{\Gamma(n/p'+1)}{\Gamma(n/2+1)} \)^{-1/p} \exp \( -\frac{1}{\sigma} |x - \bar{x}|^{p'} \)
\]
where $\sigma > 0$ and $\bar{x} \in \re^n$.
For the proof of this fact, the authors in \cite{DD} used an idea that the inequality \eqref{ELSp} can be thought of as a limiting case of a family of sharp Gagliardo-Nirenberg inequalities.
In order to characterize all extremizers of \eqref{ELSp}, they relate them to the positive ground state solutions to the equation $-\Delta_p u = u^{p-1} \log u$ in $\re^n$,
and use the radial symmetry and (highly non-trivial) uniqueness result of this equation.
Note that this approach is valid for $1 < p < n$.

Later, Van Hoang Nguyen \cite{Ng} obtained the sharp $L^p$-logarithmic Sobolev inequality \eqref{ELSp} with special weights, \textit{i.e.}, the monomial weights.
For the proof, he first established the sharp Sobolev inequality with monomial weights together with the precise information of extremal functions,
and followed the idea by Beckner and Pearson \cite{Beckner-Pearson}.
His method to establish the sharp Sobolev inequality with monomial weights (or more generally, the sharp Gagliardo-Nirenberg type inequalities with monomial weights) is to adapt the mass transportation technique 
(e.g., see \cite{CE-N-V}) and does not need any symmetrization and rearrangement. 
The resulting inequality is valid for any norm on $\re^n$ other than the Euclidean norm, which we will show below:

Let $\| \cdot \|$ be any norm on $\re^n$ and let $\| \cdot \|_*$ be its dual norm
\[
	\| \xi \|_* = \sup_{\| x \| \le 1} (x \cdot \xi)
\]
where $\cdot$ denotes the inner product on $\re^n$.
Let $A = (A_1, A_2, \dots, A_n)$ be a nonnegative vector in $\re^n$, \textit{i.e.} $A_1 \ge 0, \dots, A_n \ge 0$. 
Define
\[
	\re^n_A = \{ x = (x_1, \dots, x_n) \in \re^n \, : \, x_i > 0 \ \text{if} \ A_i > 0 \}
\]
and the monomial weight
\[
	x^A = |x_1|^{A_1} \cdots |x_n|^{A_n} \text{ for } x \in \re^n_A.
\]
For any open set $\Omega$ of $\re^n$ and $1 \le p < \infty$, let us denote by  $W^{1,p}(\Omega, x^A dx)$ the set of measurable functions on $\Omega$ such that
$\| f \|^p_{W^{1,p}(\Omega, x^A dx)} = \int_{\Omega} (|f(x)|^p +\| \nabla f(x) \|_*^p) x^A dx < \infty$.
Put $B_1=\{x: \| x \| < 1\}$ and denote $B_A = \re^n_A \cap B_1$.
For a bounded Lipschitz domain $\Omega \subset \re^n$,
let 
\[
	m(\Omega) = \intO x^A dx
\]
denote the weighted volume of $\Omega$ and
\[
	D = n + A_1 + \cdots + A_n.
\]

Then Proposition 1.4 in \cite{Ng} for $1 < p < \infty$ can be read as follows: 

\vspace{1em}\noindent
{\bf Theorem A.} {\rm (Sharp $L^p$-logarithmic Sobolev inequality with monomial weights \cite{Ng})}
{\it
Let $1 < p< +\infty$, $A = (A_1, A_2, \dots, A_n)$ be a nonnegative vector in $\re^n$ and $D=n+A_1+\cdots+A_n$.  
For any $f \in W^{1,p}(\re^n,x^A\,dx)$ such that $\int_{\re^n_A} |f|^p x^A dx = 1$,
the inequality
\begin{equation}
\label{ELSWp}
	\int_{\re^n_A} |f|^p (\log |f|^p) x^A dx \le \frac{D}{p} \log \( \mathcal{L}_{p}(A) \int_{\re^n_A} \| \nabla f \|_*^p x^A dx \)
\end{equation}
holds true, where 
\begin{equation}
\label{L}
	\mathcal{L}_{p}(A)= \frac{p}{D}\left(\frac{p-1}{e}\right)^{p-1}\left[\Gamma\(\frac{D}{p'}+1\)m(B_A)\right]^{-p/D}
\end{equation}
The equality in \eqref{ELSWp} holds if
\[
	f(x)= \beta \exp \(-\frac{\| x-x_0 \|^{p'}}{\sigma}\)
\]
for some $\sigma>0$, $x_0 = ((x_0)_1, \cdots, (x_0)_n) \in \mathbb{R}^n$ such that $(x_0)_i=0$ if $A_i>0$ for $1 \le i \le n$,
and $\beta$ must satisfy
$|\beta|^{-p}=\int_{\re^n_A} e^{-\frac{p}{\sigma} \| x - x_0 \|^{p'}}x^A\,dx$. 
}


\vspace{1em}
Finally, in a very recent paper \cite{BDK}, Balogh, Don and Krist\'aly obtain the sharp $L^p$ logarithmic Sobolev inequality for $1 \le p < \infty$ on any open cone $E \subset \re^n$ 
with log-concave homogeneous weight $\omega$ of degree $\tau \ge 0$: \textit{i.e.}, $\omega : E \to (0, \infty)$ is such that $E \ni x \mapsto \log \omega(x)$ is concave 
and $\omega(\la x) = \la^{\tau} \omega(x)$ for any $x \in E$ and $\la \ge 0$.
Let us define $L^p(E, \omega) = \{ f: E \to \re, \int_E |f|^p \omega dx < +\infty \}$ and $W^{1,p}(E, \omega) = \{ f \in L^p(E, \omega) : \| \nabla f \|_* \in L^p(E, \omega) \}$.
Then, the inequality obtained in \cite{BDK} for $1 < p < \infty$ is
\begin{equation}
\label{LS_BDK}
	\int_E |f|^p (\log |f|^p) \omega dx \le \frac{n+\tau}{p} \log \( \mathcal{L}_{\omega, p} \int_E \| \nabla f \|_*^p \omega dx \),
\end{equation}
which holds for any $f \in W^{1,p}(E,\omega)$ such that $\int_E |f|^p \omega dx = 1$,
where
\[
	\mathcal{L}_{\omega, p} = \frac{p}{n+\tau}\left(\frac{p-1}{e}\right)^{p-1}\left[\Gamma\(\frac{n+\tau}{p'}+1\) \int_{E \cap B_1} \omega dx \right]^{-p/(n+\tau)}
\]
If we consider the monomial weight $\omega(x) = x^A$, 
then $E = \re^n_A$, the homogeneous degree $\tau = A_1 + \cdots + A_n = D - n$, and $\mathcal{L}_{\omega, p} = \mathcal{L}_p(A)$. 
Thus the inequality \eqref{LS_BDK} recovers \eqref{ELSWp}.
The authors in \cite{BDK} use (again) the optimal transport theory to give a direct proof of their sharp weighted $L^p$-logarithmic Sobolev inequality, 
avoiding the use of any limiting argument in the Gagliardo-Nirenberg or Sobolev type inequality, as typical proofs do.
The perk of this method is to derive the complete ``if and only if" characterization of the equality case of their new weighted sharp $L^p$-logarithmic Sobolev inequality \eqref{LS_BDK}.

In this paper we propose a new (even in the unweighted case) and an elementary proof of the characterization of the equality case for \eqref{ELSWp} with $p=2$, when the norm is standard Euclidean norm $| \cdot |$
(Theorem \ref{Theorem:equality}).
By the same method, we can prove the equality case for \eqref{ELSWp} for $1 < p < \infty$ with an arbitrary norm, but the unweighted case $A = (0, \cdots, 0)$ (Theorem \ref{Theorem:equality_p}). 
Though Theorem \ref{Theorem:equality} and Theorem \ref{Theorem:equality_p} below in this paper are completely covered by the result in \cite{BDK},
however, we believe that the method we propose here is still interesting and will shed light on the hidden analytic features of the $L^p$-logarithmic Sobolev inequality.

Define
\begin{equation}
\label{Pi}
	\Pi(A)= \left[ \frac{\prod_{i=1}^n \Gamma(\frac{A_i+1}{2})}{2^k} \right]^{2/D}
\end{equation}
where  $k = \sharp \{ i \in \{ 1, \dots, n \} \, : \, A_i > 0 \}$ denotes the number of positive entries of the vector $A$.
The result below is first proved by \cite{DD} when $A = (0, \cdots, 0)$. 

\begin{theorem}\label{Theorem:equality}
Let $A = (A_1, A_2, \dots, A_n)$ be a nonnegative vector in $\re^n$ and $D=n+A_1+\cdots+A_n$. 
Then the equality in the sharp $L^2$-logarithmic Sobolev inequality with monomial weight
\begin{equation}
\label{ELSW2}
	\int_{\re^n_A} |f|^2 (\log |f|^2) x^A dx \le \frac{D}{2} \log \( \frac{2}{\Pi(A)eD} \int_{\re^n_A} |\nabla f|^2 x^A dx \),
\end{equation}
which holds for any $f \in W^{1,2}(\re^n,x^A\,dx)$ such that $\int_{\re^n_A} |f|^2 x^A dx = 1$,
occurs if and only if 
\[
	f(x) = \beta \exp \(-\frac{|x-x_0|^2}{\sigma}\)
\]
for some $\sigma>0$, $x_0 = ((x_0)_1, \cdots, (x_0)_n)\in \mathbb{R}^n$ such that $(x_0)_i=0$ if $A_i>0$, 
and $\beta \in \re$ satisfies
\[
	|\beta|^{-2}= \( \frac{\sigma \Pi(A)}{2} \)^{\frac{D}{2}}.
\]
\end{theorem}

\vspace{1em}
In order to explain the basic idea of the proof let us consider the unweighted case $A = (0, \dots, 0)$ in \eqref{ELSW2}.
We take into account the following observations:

\begin{itemize}
\item[i)] The logarithmic Sobolev inequality can be obtained (see the sketch of the proof of Theorem A below) as a ``limit" of the Sobolev inequality for suitable functions.
\item[ii)] The equality case in the classical Sobolev inequality occurs if and only if the functions are of the form
\begin{equation}\label{classic equality}
	a^{-1}(1 +b|x - x_0|^{p'})^{1-\frac{n}{p}}
\end{equation}
where $a\in\mathbb{R}-\{0\}$ and $b>0$ and $x_0\in\mathbb{R}^n$, see \cite{Talenti}.
\item[iii)] When $p=2$, the Euclidean $L^2$-logarithmic Sobolev inequality is known to be equivalent to Gross's logarithmic Sobolev inequality \cite{Gross}
\[
	\int_{\re^n} |g|^2 \log |g| d\gamma \le \int_{\re^n} |\nabla g|^2 d\gamma
\]
for $g$ satisfying $\int_{\re^n} |g|^2 d\gamma = 1$,  where $d\gamma(x) = (2\pi)^{-n/2} e^{-|x|^2/2} dx$ denotes the Gaussian probability measure.
\item[iv)] When $p' = 2$, the family of functions \eqref{classic equality} are densities of some generalized Cauchy distributions,
which have the general form $a^{-1}(1 +b|x|^{2})^{-\beta}$ with $b,\beta>0$ and normalizing constant $a = \int_{\re^n} (1 + b|x|^2)^{-\beta} dx$.
From the observations by Bobkov and Ledoux \cite{Bob},
these probability measures may be considered as a natural \textquotedblleft pre-Gaussian model\textquotedblright,
where the Gaussian case appears in the limit as $\beta\rightarrow + \infty$ (after proper rescaling of the coordinates).
Actually, we can check by using central limit theorem that the normalized Cauchy distribution approaches to the Gaussian distribution as $\beta \to +\infty$.
Note that in order to consider Talenti functions \eqref{classic equality} as generalized Cauchy distributions, the exponent must satisfy $p' = 2$. 
\end{itemize}
Similar observations hold even in the case $A \not\equiv (0, \cdots, 0)$.

Though our observation does not match to the probabilistic interpretation when $p \ne 2$, 
fortunately,
the same kind of strategy does work when the inequality is unweighted. 
Also in this unweighted case we can allow that the inequality may involve any norm in $\re^n$.
See Theorem \ref{Theorem:equality_p} in \S 5.

Note that the authors proposed this idea of proving the characterization of the equality cases first in an unpublished arXiv paper \cite{Feo-TF} in 2019.

The rest of the paper is organized as follows:
In \S 2, we collect some useful lemmas.
In \S 3, we sketch a proof of Theorem A in \cite{Ng} only when the exponent $p=2$, which is needed in the proof of Theorem \ref{Theorem:equality}.
\S 4 is devoted to the proof of Theorem \ref{Theorem:equality}.
\S 5 is devoted to the proof of Theorem \ref{Theorem:equality_p}.

The equality case of logarithmic Sobolev inequalities is studied from many different points of view, see, for example, \cite{Ohta-Takatsu}.

\section{Computational lemmas}

First we collect here several lemmas which will be useful.
Next lemma is an exercise of the book by W. Rudin \cite{Rudin} Chapter 3, page 71.

\begin{lemma}\label{Lemma0}
Let $A = (A_1, A_2, \dots, A_n)$ be a nonnegative vector in $\re^n$ and $D=n+A_1+\cdots+A_n$. 
Let $(X, \mu)$ be a measure space with $\mu(X) = 1$ and assume that $g \in L^s(X, \mu)$ for some $s > 0$.
Then it holds
\[
	\lim_{t \to +0} \( \int_X |g|^t d\mu \)^{1/t} = \exp \( \int_X \log |g(x)| d\mu \)
\]
if $\exp(-\infty)$ is defined to be $0$.
\end{lemma}

\begin{lemma}\label{Lemma1}
Let $x_0 = ((x_0)_1, \cdots, (x_0)_n) \in \re^n$ be such that $(x_0)_i = 0$ if $A_i > 0$ for any $i$. 
Then for any $\al > 0$ and $t > 0$, we have
\begin{align}
	&\int_{\re^n_A} e^{-t|x - x_0|^{\al}} x^A dx = t^{-\frac{D}{\al}} \frac{\Gamma(\frac{D}{\al}+1)}{\Gamma(\frac{D}{2}+1)} \Pi(A)^{\frac{D}{2}}, \label{G1} \\
	&\int_{\re^n_A} e^{-t|x - x_0|^{\al}} |x - x_0|^{\al} x^A dx = \frac{D}{\al} t^{-\frac{D}{\al}-1} \frac{\Gamma(\frac{D}{\al}+1)}{\Gamma(\frac{D}{2}+1)} \Pi(A)^{\frac{D}{2}}. \label{G2}
\end{align}
\end{lemma}

\begin{proof}
Since \eqref{G2} is derived from \eqref{G1} by differentiating it with respect to $t$, we prove \eqref{G1} only.
Put $x - x_0 = y$. Then we check that $x \in \re^n_A$ is equivalent to $y \in \re^n_A$ and
\begin{align*}
	x^A dx &= |x_1|^{A_1} \cdots |x_n|^{A_n} dx_1 \cdots dx_n \\
	&= |y_1 + (x_0)_1|^{A_1} \cdots |y_n + (x_0)_n|^{A_n} dy_1 \cdots dy_n \\
	&= |y_1|^{A_1} \cdots |y_n|^{A_n} dy_1 \cdots dy_n \\
	&= y^A dy
\end{align*}
since $(x_0)_i = 0$ if $A_i > 0$.
Thus it is enough to prove \eqref{G1} when $x_0 = (0, \cdots, 0)$.

Let $B_A = \re^n_A \cap B_1$ 
and put $x = (x_1, \cdots, x_n) = r \omega$, where $r = |x|$ and $\omega = (\omega_1, \cdots, \omega_n)$ be the unit vector in $\pd B_A = \pd B_1 \cap \re^n_A$.
Note that $\pd B_A$ is the curved part of the boundary portion of $B_A$.
Then $x^A = (r\omega)^A = (r\omega_1)^{A_1} \cdots (r\omega_n)^{A_n}$ and
\[
	x^A dx = (r\omega)^A r^{n-1} dr dS_\omega = r^{A_1 + \cdots + A_n + n-1} \omega^A \,dr \,dS_{\omega},
\]
where $dS_{\omega}$ denotes the surface measure on
 $\pd B_A$.
We calculate
\begin{align*}
	\int_{\re^n_A} e^{-|x|^{\al}} x^A dx &= \int_{\pd B_A} \int_0^{\infty} e^{-r^\al} (r \omega)^A r^{n-1} dr dS_{\omega} \\
	&= \( \int_{\pd B_A} \omega^A dS_{\omega} \) \int_0^{\infty} e^{-r^\al} r^{D-1} dr \\
	&= P(B_A) \frac{1}{\al} \Gamma\(\frac{D}{\al}\),
\end{align*}
where $P(B_A) = \int_{\pd B_A} \omega^A dS_{\omega}$.
As observed in \cite{Cabre-RosOton} (Theorem 1.4 and Lemma 4.1), $P(B_A) = D m(B_A)$ and 
\[
	m(B_A) = \int_{B_A} x^A dx = \frac{\Pi(A)^{\frac{D}{2}}}{\Gamma(\frac{D}{2}+1)}.
\]
Thus we obtain
\begin{align*}
	\int_{\re^n_A} e^{-|x|^{\al}} x^A dx = \frac{\Gamma\(\frac{D}{\al} + 1\)}{\Gamma\(\frac{D}{2}+1\)} \Pi(A)^{\frac{D}{2}}.
\end{align*}
The transformation $t^{1/\al} x = y$ for $x \in \re^n_A$ yields \eqref{G1}.
\end{proof}

\begin{lemma}\label{Lemma2}
Let $\al>1,\sigma>0$.
Let $x_0 = ((x_0)_1, \cdots, (x_0)_n) \in \re^n$ be such that $(x_0)_i = 0$ if $A_i > 0$ for $1 \le i \le n$.  
Then we have
\begin{equation}
\label{C1}
	\int_{\re^n_A} \frac{1}{(1 +\sigma |x - x_0|^{\al})^{\beta}}\,x^A \,dx = \frac{2}{\al}\frac{\Pi(A)^{\frac{D}{2}}}{\sigma^{\frac{D}{\al}}}\frac{\Gamma(\frac{D}{\al})}{\Gamma(\frac{D}{2})} \frac{\Gamma(\beta - \frac{D}{\al})}{\Gamma(\beta)} \quad \hbox{ for } \al \beta > D.
\end{equation}
for any $x_0 \in \re^n$.
\end{lemma}

\begin{proof}
As in the same reason in the preceding lemma, we may consider only the case $x_0 = (0, \cdots, 0)$.
Then we derive \eqref{C1} by using the ``polar coordinates" again.
As in the proof of the former lemma, we compute
\begin{align*}
	&\int_{\re^n_A} \frac{x^A dx}{(1 +\sigma |x|^\al)^{\beta}} = \int_{\pd B_A} \int_0^{\infty} \frac{(r \omega)^A r^{n-1} dr dS_{\omega}}{(1 + \sigma r^\al)^{\beta}} \\
	&= \( \int_{\pd B_A} \omega^A dS_{\omega} \) \int_0^{\infty} \frac{r^{D-1}}{(1 + \sigma r^\al)^{\beta}} dr
	= P(B_A) \(\frac{1}{\sigma} \)^{\frac{D}{\al}} \int_0^{\infty} \frac{s^{D-1}}{(1 + s^\al)^{\beta}} ds \\
	&= D \frac{\Pi(A)^{\frac{D}{2}}}{\Gamma(\frac{D}{2}+1)} \(\frac{1}{\sigma} \)^{\frac{D}{\al}} \frac{\Gamma(\frac{D}{\al}) \Gamma(\beta - \frac{D}{\al})}{\al\Gamma(\beta)}
	= \frac{2}{\al}\frac{\Pi(A)^{\frac{D}{2}}}{\sigma^{\frac{D}{\al}}}\frac{\Gamma(\frac{D}{\al})}{\Gamma(\frac{D}{2})} \frac{\Gamma(\beta - \frac{D}{\al})}{\Gamma(\beta)},
\end{align*}
which proves lemma.
\end{proof}

\section{Sketch of the proof of Theorem A for $p=2$.}

We will give some details of proof of Theorem A in \cite{Ng}
because we need them in what follows.

By density argument it is enough to prove the result for functions $f\in C_c^1(\re^n)$.

In the paper by Cabr\'e and Ros-Oton \cite{Cabre-RosOton}, 
the Sobolev inequality with monomial weights is proved, and in \cite{Ng} the inequality is generalized to any norm on $\re^n$.
When we focus on the Euclidean norm and $1 < p < D$ case,
the Sobolev inequality with monomial weights proved in \cite{Cabre-RosOton} and \cite{Ng} reads as follows:

\vspace{1em}\noindent
{\bf Theorem B.}{\rm (Sharp $L^p$-Sobolev inequality with monomial weights \cite{Cabre-RosOton}, \cite{Ng})}
{\it
Let $A = (A_1, A_2, \dots, A_n)$ be a nonnegative vector in $\re^n$, $D=n+A_1+\cdots+A_n$, $1 < p < D$ and $p_* = \frac{Dp}{D-p}$.
Then the inequality
\begin{equation}\label{Sobolev weight}
	\( \int_{\re^n_A} |f|^{p_{*}}\,x^A\, dx \)^{1/p_{*}} \le C_{p,n,A} \( \int_{\re^n_A} |\nabla f|^p \,x^A \,dx \)^{1/p}
\end{equation}
holds true for any $f \in W^{1,p}(\re^n, x^A dx)$, where
\begin{equation}
\label{CpnA}
	C_{p,n,A}= D^{-\frac{1}{p}-\frac{1}{D}} \( \frac{p-1}{D-p} \)^{\frac{1}{p^{\prime}}} \( \frac{p^{\prime}\Gamma(D)}{\Gamma\(\frac{D}{p}\)\Gamma\(\frac{D}{p^{\prime}}\) m(B_A)} \)^{\frac{1}{D}}.
\end{equation}
Define
\[
	h_{p,A}(x) = (\sigma_{p,A} + |x|^{p'})^{1-D/p} 
\]
where
$\sigma_{p,A}$ is the normalized constant such that $\int_{\re^n_A} |h_{p,A}|^{p_{*}}\,x^A\, dx=1$.
Then the equality in \eqref{Sobolev weight} for $1 < p < D$ holds if and only if 
\[
	f(x)=ch_{p,A}(\lambda(x-x_0))
\]
where $c\in\mathbb{R}$,$\lambda>0$, and $x_0 = ((x_0)_1, \cdots, (x_0)_n)\in \mathbb{R}^n$ such that $(x_0)_i=0$ if $A_i>0$.

}


\vspace{1em}
From now on, we focus only on the case $p = 2$.

Let $N \in \N$ and let $B = (B_1, \dots, B_N) \in \re^N$ be a nonnegative vector.
Let us take $F \in C_c^1(\re^N)$ with $\int_{\re^N_B} |F(z)|^2 z^B dz = 1$ and let us denote $D_B = N + B_1 + \cdots + B_N$ and $2_*(B) = \frac{2D_B}{D_B-2}$.
Then the sharp $L^2$-Sobolev inequality \eqref{Sobolev weight} for $F$ yields that
\begin{equation}
\label{L2}
	\( \!\int_{\re^N_B} |F(z)|^{2_*(B)} z^B dz \!\)^{\!\frac{1}{2_*(B)-2}}\!\!\! \le \! \(\! C_{2,N,B}^2 \int_{\re^N_B} |\nabla F(z)|^2 z^B dz \!\)^{\!\frac{D_B}{4}}\!\!\!.\!\!
\end{equation}
Let us consider a nonnegative vector $A = (A_1, \dots, A_n) \in \re^n$. We express $N = l n$ for $l \in \N$.
For a function $f \in C_c^1(\re^n)$ satisfying $\int_{\re^n_A} |f(x)|^2 x^A dx = 1$,
we put
\[
	B = (\underbrace{A, A, \dots, A}_{l}) \in \re^{l n} = \re^N \quad \text{ and }\quad F(z) = \prod_{i=1}^l f(x^i),
\]
where $x^i = (x^i_1, \dots, x^i_n) \in \re^n$ for each $i = 1,2,\dots,l$,
and $z = (x^1, x^2, \cdots, x^l) \in \re^{ln} = \re^N$.
By a direct computation we get
\begin{align}
	&\int_{\re^N_B} |F(z)|^t z^B dz = \prod_{i=1}^l \int_{\re^n_A} |f(x^i)|^t (x^i)^A dx^i, \quad \forall t \ge 1, \label{R1} \\
	&\int_{\re^N_B} |\nabla F(z)|^2 z^B dz = l \int_{\re^n_A} |\nabla f(x)|^2 x^A dx. \label{R2}
\end{align}
Also note that $D_B = l D$.
Then \eqref{L2} becomes, after taking $1/l$-th root of both sides,
\begin{equation}\label{51}
	\( \int_{\re^n_A} |f(x)|^{2_*(B)-2} |f(x)|^2 x^A dx \)^{\frac{1}{2_*(B)-2}} \le \( l C_{2,N,B}^2 \int_{\re^n_A} |\nabla f(x)|^2 x^A dx \)^{D/4}.
\end{equation}
Now, by \eqref{CpnA} and $D_B = lD$, we have
$$
	l C_{2,N,B}^2= \frac{1}{D} \(\frac{1}{\Pi(A)}\)^{2/D} \frac{1}{lD-2} \left[ \frac{\Gamma(l D)}{\Gamma(\frac{l D}{2})} \right]^{\frac{2}{l D}}.
$$
Let $l \to \infty$ in the above equality.
Stirling's formula
\begin{equation}
\label{Stirling}
	\Gamma(s) = [1 + o(1)] \(\sqrt{2\pi} s^{s-1/2} e^{-s} \)  \quad (s \to \infty)
\end{equation}
implies that
\begin{equation}\label{C l}
	\lim_{l \to \infty} l C_{2, ln, B}^2 = \mathcal{L}_2(A)
\end{equation}
where $\mathcal{L}_2(A) = \frac{2}{\Pi(A) e D}$ is defined in \eqref{L}.
Note that $2^*(B)-2 = \frac{4}{lD -2} \to 0$ as $l \to \infty$.
Then Theorem A follows by taking a limit $l \to \infty$ in \eqref{51} using Lemma \ref{Lemma0} with $d\mu = |f(x)|^2 x^A dx$ and \eqref{C l}.
\qed


\begin{remark}
In the previous proof if we consider the general exponent $1 < p < \infty$, then
\eqref{R1} still holds true but \eqref{R2} (with the exponent $2$ replaced by $p$) does not hold for $p\neq2$. 
In order to overcome this difficulty, the author in \cite{Ng} first establishes as a starting point 
the sharp Sobolev inequality that involves a general norm with the exact information of extremals,
and uses \eqref{R2} having different norms on both sides; see \cite{Ng} for details (or compare \eqref{R1}, \eqref{R2} to \eqref{R1_p}, \eqref{R2_p} in \S 5.) 
\end{remark}

\begin{remark}
By Lemma \ref{Lemma1} with $\alpha = 2$ and $t = 1/2$, 
we easily check that the equality in \eqref{ELSW2} holds for $f(x)=\frac{e^{-\frac{|x|^2}{4}}}{\left(2\Pi(A)\right)^{\frac{D}{4}}}$
which satisfies that $\int_{\re^n_A} |f|^2 x^A\,dx=1$ and $\int_{\re^n_A} |x|^2|f|^2 x^A\,dx=D$.
\end{remark}


\section{Proof of Theorem \ref{Theorem:equality}}

In this section, we prove Theorem \ref{Theorem:equality}. 
As explained in Introduction, our basic idea stems from the probabilistic point of view, see \cite{Bob} 

It is easy to check that functions of the form $\beta e^{-\frac{|x-x_0|^2}{\sigma}}$ for some $\sigma>0$, $x_0\in \mathbb{R}^n$ such that the component $(x_0)_i=0$ if $A_i>0$, 
give the equality in the $L^2$-logarithmic Sobolev inequality \eqref{ELSW2}.
We want to prove that they are the only extremals.
In order to do so we characterize the equality cases in every inequality in the proof of Theorem A.
Without loss of generality we may consider only positive functions.
Recall that extremals in the classical Sobolev inequality are all given by \eqref{classic equality}, see \cite{Talenti}.
In the following, we will use the notation $\alpha_l \sim \beta_l$ if $\lim_{l \to \infty} \frac{\alpha_l}{\beta_l} = 1$.
In order to have the equality for some $F_l(z)$ in \eqref{L2},
it is necessary that, as $l \to \infty$,
\begin{equation}
\label{start}
	F_l(z) \sim a_l(1 + b_l|z-z_0^l|^2)^{1-\frac{D_B}{2}} \quad \text{ for every } z\in\mathbb{R}^N_B
\end{equation}
with some $a_l>0$, $b_l>0$, and $z_0^l = ((z_0^l)_1, \cdots, (z_0^l)_N) \in \re^N$ such that $(z_0^l)_i = 0$ if $B_i > 0$ for all $1 \le i \le N$.
Here we have used the notation which emphasizes the dependence on $l$ of involved functions and constants.
By translation invariance, we may fix $z_0^l = z_0$ for a fixed point $z_0 = ((z_0)_1, \cdots, (z_0)_N) \in \re^N$ such that $(z_0)_i = 0$ if $B_i > 0$.
Also recalling that we consider function with $\int_{\re^N_B} |F_l(z)|^2 z^B dz = 1$, 
we have by \eqref{start} and Lemma \ref{Lemma2},
\begin{align}
\label{41}
	a_l^2 \int_{\re^N_B} (1 + b_l |z-z_0|^2)^{2-D_B} z^B dz &=  a_l^2 \( \frac{\Pi(B)}{b_l} \)^{D_B/2} \frac{\Gamma(D_B/2-2)}{\Gamma(D_B-2)} \sim 1
\end{align}
as $l \to \infty$.
Note that $\Pi(B) = \Pi(A)$ by definition \eqref{Pi} and $D_B = l D$.
Thus by \eqref{41} and Stirling's formula \eqref{Stirling}, we see
\begin{align}
	(a_l)^{1/l} &\overset{\eqref{41}}{\sim}  \( \frac{b_l}{\Pi(B)} \)^{\frac{D}{4}} \( \frac{\Gamma(D_B-2)}{\Gamma(D_B/2-2)} \)^{\frac{1}{2l}} \notag \\
	&\overset{\eqref{Stirling}}{\sim}  \( \frac{b_l}{\Pi(A)} \)^{\frac{D}{4}} \( e^{-\frac{lD}{2}} 2^{(lD-5)/2} (lD)^{\frac{lD}{2}} \)^{\frac{1}{2l}} \notag \\
	&\sim  \( \frac{2 D}{e \Pi(A)} l b_l \)^{\frac{D}{4}} \quad (l \to \infty). \label{all}
\end{align}
Recall $F_l(z) = \Pi_{i=1}^l f_l(x^i)$ where $z = (x^1, \cdots, x^l) \in \re^{nl} = \re^N$, $x^i \in \re^n$ $(i=1,\cdots, l)$.
We choose
$$
	z=(x,\cdots,x), \quad z_0=(x_0,\cdots,x_0) \quad \text{with } x,x_0\in\re^n.
$$
Then it follows that $|z - z_0| = \sqrt{l} |x - x_0|$ and
\begin{equation}
\label{42}
	f_l(x) = (F_l(z))^{1/l} \overset{\eqref{41}}{\sim} (a_l)^{1/l} (1 + l b_l |x-x_0|^2)^{1/l-D/2} \quad \forall x \in \re^n
\end{equation}
as $l \to \infty$.
Note that $(f_l(x_0))^l = F_l(z_0) \sim a_l$ as $l \to +\infty$.

We have three possible behaviors of the sequence $b_l$ as $l \to \infty$:
\begin{itemize}
\item [i)] $b_l\rightarrow+\infty,$
\item [ii)] $b_l\rightarrow \bar{b}\in \mathbb{R}-\{0\},$
\item [iii)]$b_l\rightarrow 0.$
\end{itemize}
Indeed if the limit does not exist, then we can argue one of the previous cases up to a subsequence.

If $b_l \to +\infty$, we have the contradiction since 
\begin{align*}
	f_l(x) &\overset{\eqref{42}}{\sim} (a_l)^{1/l} (1 + l b_l |x-x_0|^2)^{1/l-D/2} \\
	&\overset{\eqref{all}}{\sim} \( \frac{2 D}{e \Pi(A)} l b_l \)^{\frac{D}{4}} (1 + l b_l |x-x_0|^2)^{1/l-D/2} \\
	&\sim (b_l)^{-D/4 + 1/l} \ l^{-D/4 + 1/l} \( \frac{2 D}{e \Pi(A)} \)^{\frac{D}{4}} \to 0
\end{align*}
as $l \to +\infty$ by \eqref{all} for any $x \in \re^n$,
which is absurd by the restriction $\int_{\re^n} |f_l(x)|^p dx = 1$.

Also if $b_l \to \bar{b} \in (0, +\infty)$,
\begin{align*}
	f_l(x) 
	&\sim \( \frac{2 D}{e \Pi(A)} l b_l \)^{\frac{D}{4}} (1 + l b_l |x-x_0|^2)^{1/l-D/2} \\
	&\sim (\bar{b})^{-D/4} l^{-D/4 + 1/l}  \( \frac{2 D}{e \Pi(A)}\)^{D/4} \to 0
\end{align*}
as $l \to +\infty$ by \eqref{all} for any $x \in \re^n$, again a contradiction.

Thus the only possible case is the third one and we have $b_l \to 0$ as $l \to \infty$.

Next, we choose
$$
	z=(x,\underbrace{x_0 \cdots,x_0}_{l-1}), \quad z_0=(x_0,\cdots,x_0) \quad \text{with } x,x_0\in\re^n.
$$
Then it follows that $|z - z_0| = |x - x_0|$ and since $F_l(z) = f_l(x) (f_l(x_0))^{l-1} \sim (a_l)^{\frac{l-1}{l}} f_l(x)$, 
we have by \eqref{start} that
\begin{equation}
\label{43}
	f_l(x) \sim (a_l)^{1/l} (1 + b_l |x-x_0|^2)^{1-D_B/2}, \quad \forall x \in \re^n
\end{equation}
as $l \to \infty$.
In this case, by $\log(1 + t) = t + o(1)$ as $t \to +0$, we see
\begin{align}
	f_l(x) &\overset{\eqref{43}}{\sim} (a_l)^{1/l} (1 + b_l |x-x_0|^2)^{1-D_B/2} \notag \\
	&= (a_l)^{1/l} \exp \left\{ (1 - D_B/2) \log (1 +  b_l |x-x_0|^2) \right\} \notag \\
	&\sim (a_l)^{1/l} \exp \left\{ (1 - D_B/2)   b_l |x-x_0|^2) \right\}. \label{44}
\end{align}

Again three cases (up to a subsequence) are possible for the behaviors of the sequence $lb_l$,
\begin{enumerate}
\item[ \textbf{(iii-1)}] $lb_l\rightarrow+\infty$
\item[ \textbf{(iii-2)}] $lb_l\rightarrow 0$
\item[ \textbf{(iii-3)}] $lb_l\rightarrow \widetilde{b}\in \mathbb{R}-\{0\}.$
\end{enumerate}

The first case (iii-1) does not occur since by \eqref{all} and $D_B = lD$, we have
\begin{align*}
	f_l(x) &\overset{\eqref{44}}{\sim} (a_l)^{1/l} \exp \( (1 - D_B/2)   b_l |x-x_0|^2) \) \\
	&\overset{\eqref{all}}{\sim}  \( \frac{2 D}{e \Pi(A)} l b_l \)^{\frac{D}{4}} \exp \( (1/l-D/2) l b_l |x-x_0|^2) \) \to 0
\end{align*}
as $l \to +\infty$ for any $x \in \re^n$ if $l b_l \to +\infty$.

Also the second case (iii-2) does not occur, since if it would happen, then
$a_l^{1/l} \overset{\eqref{all}}{\sim}  \( \frac{2 D}{e \Pi(A)} l b_l \)^{\frac{D}{4}} \to 0$ as $l \to \infty$ and
\begin{align*}
	f_l(x) \sim (a_l)^{1/l} (1 + b_l |x-x_0|^2)^{1-(lD)/2} \to 0
\end{align*}
as $l \to +\infty$ for any $x \in \re^n$.

Thus the only case to be considered is the third one: $lb_l\rightarrow \widetilde{b}\in \mathbb{R}-\{0\}$.
In this case, we have
\begin{align*}
	f_l(x) &\overset{\eqref{44}}{\sim} (a_l)^{1/l} \exp \( (1 - D_B/2)   b_l |x-x_0|^2) \) \\
	&\overset{\eqref{all}}{\sim}  \( \frac{2 D}{e \Pi(A)} l b_l \)^{\frac{D}{4}} \exp \( (1/l -D/2)  l b_l |x-x_0|^2) \) \\
	&\sim  \( \frac{2 D}{e \Pi(A)} \bar{b} \)^{\frac{D}{4}} \exp \( (-D/2)  \bar{b} |x-x_0|^2) \)
\end{align*}
as $l \to +\infty$. This is the desired conclusion.

\qed

\section{On the equality case for the unweighted sharp $L^p$-logarithmic Sobolev inequality with arbitrary norm.}

In this section, we characterize the equality case for \eqref{ELSWp} for $1 < p < \infty$ with an arbitrary norm, but the unweighted case $A = (0, \cdots, 0)$.
We fix $n \in \N$.
Let $\| \cdot \|$ be any norm on $\re^n$ and let $\| \cdot \|_*$ be its dual norm.
Put $B^n = \{ x \in \re^n : \| x \| < 1 \}$.
We put $N = nl$ for some $l \in \N$.
Later, we let $l$ tend to $\infty$.
Let $1 < p < \infty$.
For $z = (x^1, \cdots, x^l) \in (\re^n)^l = \re^N$, $x^i \in \re^n$ for $1 \le i \le l$, define a new norm as in \cite{Ng}
\[
	||| z ||| = \( \sum_{i=1}^l \| x^i \|^{p'} \)^{1/{p'}}
\]
where $1/p + 1/p' =1$.
Let $\tilde{B}^N = \{ z \in \re^N : |||z||| < 1 \}$ be the unit ball in $\re^N = \re^{nl}$ with respect to the norm $||| \cdot |||$,
and let $B^n = \{ x \in \re^n : \| x \| < 1 \}$ be the unit ball in $\re^n$ with respect to the norm $\| \cdot \|$, respectively.

\vspace{1em}
As in the proof of Theorem \ref{Theorem:equality}, we start from some computational lemmas.

\begin{lemma}\label{Lemma5.1}
Let $N = nl$ and put $A_l = \int_{\tilde{B}^N} 1 dz$ be the volume of $\tilde{B}^N = \{ z \in \re^N : |||z||| < 1 \}$. 
Then we have
\begin{equation}
\label{Al}
	A_l = \frac{(\frac{n}{p'})^{l-1} (m(B^n))^l (\Gamma(\frac{n}{p'}))^l}{l \Gamma(\frac{nl}{p'})}
\end{equation}
for $l \in \N$, where $m(B^n) = \int_{B^n} 1 dx$ is the volume of $B^n = \{ x \in \re^n : \| x \| < 1 \}$. 
\end{lemma}

\begin{proof}
When $l = 1$, then $N = n$ and $A_1 = \int_{B^n} 1 dx = m(B^n)$. 
For $l > 1$, we write $z = (x^1, x^2, \cdots, x^l)$, $x^i \in \re^n$ $(i=1, \cdots l)$, as $z = (z', x^l)$ where $z' = (x^1, \cdots, x^{l-1}) \in \re^{n(l-1)}$.
Then we have $||| z |||^{p'} = ||| z' |||^{p'} + \| x^l \|^{p'}$, where $||| z' ||| = \( \sum_{i=1}^{l-1} \| x^i \|^{p'} \)^{1/{p'}}$.
We compute
\begin{align*}
	A_l = \int_{||| z |||^{p'} < 1} dz &= \int_{x^l \in B^n} \( \int_{\sum_{i=1}^{l-1} \| x^i \|^{p'} < 1 - \| x^l \|^{p'}} dx^1 dx^2 \cdots dx^{l-1} \) dx^l \\
	&= \int_{x^l \in B^n} \( \int_{||| z' |||^{p'} < 1 - \| x^l \|^{p'}} dz' \) dx^l \\
	&= \int_{x^l \in B^n} \int_{||| w' |||^{p'} < 1} (1-\| x^l \|^{p'})^{\frac{n(l-1)}{{p'}}} dw' dx^l \\
	&= \( \int_{||| w' |||^{p'} < 1} dw' \) \int_{x^l \in B^n} (1-\| x^l \|^{p'})^{\frac{n(l-1)}{{p'}}} dx^l \\
	&= A_{l-1} \int_{x^l \in B^n} (1-\| x^l \|^{p'})^{\frac{n(l-1)}{{p'}}} dx^l,
\end{align*}
where we have used a change of variables
\[
	w' = \frac{z'}{(1 - \| x^l \|^{p'})^{1/{p'}}} \in \re^{n(l-1)}, \quad  dw' = \( \frac{1}{(1 - \| x^l \|^{p'})^{1/{p'}}} \)^{n(l-1)} dz'.
\]
Write $x^l = x \in \re^n$ and let us compute
\begin{align*}
	\int_{x \in B^n} (1-\| x \|^{p'})^{\frac{n(l-1)}{{p'}}} dx.
\end{align*}
Since the integrand is $\| \cdot \|$-symmetric function, we can exploit the ``polar coordinate" to get
\begin{align*}
	\int_{x \in B^n} (1-\| x \|^{p'})^{\frac{n(l-1)}{{p'}}} dx = P_{\| \cdot \|}(B^n) \int_0^1 (1 - r^{p'})^{\frac{n(l-1)}{{p'}}} r^{n-1} dr,
\end{align*}
where
\begin{align*}
	P_{\| \cdot \|}(B^n) = n \cdot m(B^n) = n \int_{\{ x \in \re^n : \| x \| < 1 \}} dx
\end{align*}
is the anisotropic perimeter with respect to the norm $\| \cdot \|$.
Thus we can continue as
\begin{align*}
	\int_{x \in B^n} (1-\| x \|^{p'})^{\frac{n(l-1)}{p'}} dx &= n m(B^n) \int_0^1 (1 - r^{p'})^{\frac{n(l-1)}{p'}} r^{n-1} dr \\
	&= \frac{n m(B^n)}{p'} \int_0^1 (1 - t)^{\frac{n(l-1)}{{p'}}} t^{\frac{n}{{p'}}-1} dt \\
	&= \frac{n m(B^n)}{p'} B \( \frac{n}{{p'}}, \frac{n(l-1)}{{p'}}+1 \) \\
	&= \frac{n m(B^n)}{p'} \frac{\Gamma(\frac{n}{p'}) \Gamma(\frac{n(l-1)}{p'}+1)}{\Gamma(\frac{nl}{p'} + 1)} \\
	&= \frac{n m(B^n)}{p'} \( \frac{l-1}{l} \) \frac{\Gamma(\frac{n}{{p'}}) \Gamma(\frac{n(l-1)}{p'})}{\Gamma(\frac{nl}{p'})},
\end{align*}
where $B(x,y)$ is the Beta function, and we have used a change of variables $r^{p'} = t$, $p'r^{p'-1} dr = dt$.
Thus we obtain
\[
	\begin{cases}
	&A_l = A_{l-1} \cdot \frac{n m(B^n)}{p'} \( \frac{l-1}{l} \) \frac{\Gamma(\frac{n}{p'}) \Gamma(\frac{n(l-1)}{p'})}{\Gamma(\frac{nl}{p'})}, \quad (l=2,3,\cdots) \\
	&A_1 = m(B^n).
	\end{cases}
\]
Put
\[
	B_l = l \cdot \Gamma(\tfrac{nl}{{p'}}) A_l, \quad (l \in \N)
\]
and $C(n,{p'}) = \frac{n m(B^n)}{{p'}} \Gamma(\frac{n}{{p'}})$.
Then we get
\[
	\begin{cases}
	&B_l = C(n,{p'}) B_{l-1}, \quad (l=2,3,\cdots) \\
	&B_1 = \Gamma(\frac{n}{{p'}}) m(B^n).
	\end{cases}
\]
Thus we have
\begin{align*}
	B_l &= (C(n,{p'}))^{l-1} B_1 = \(\frac{n}{{p'}}\)^{l-1} \( m(B^n) \Gamma\(\frac{n}{{p'}}\) \)^l,  \\
	A_l &= \frac{B_l}{l \Gamma(\frac{nl}{{p'}})} = \frac{(\frac{n}{{p'}})^{l-1} (m(B^n))^l (\Gamma(\frac{n}{{p'}}))^l}{l \Gamma(\frac{nl}{{p'}})}.
\end{align*}
	
\end{proof}

Corresponding to Lemma \ref{Lemma2}, we have the following. 

\begin{lemma}\label{Lemma5.2}
Let $\sigma > 0$ and let $\al, \beta > 0$ be such that $\al \beta > N$.
Then we have
\begin{align*}
	\int_{\re^N} \frac{1}{(1 +\sigma ||| z - z_0 |||^\al)^{\beta}} \,dz = N A_l \frac{\Gamma(\frac{N}{\al}) \Gamma(\beta - \frac{N}{\al})}{\al \Gamma(\beta)} \sigma^{-\frac{N}{\al}}
\end{align*}
for any $z_0 \in \re^N$.
\end{lemma}

\begin{proof}
It is enough to assume $z_0 = 0 \in \re^N$.
In this case, since the integrand is a symmetric function with respect to the norm $||| \cdot |||$ on $\re^N$,
we derive the formula  by using the ``polar coordinates" again.
Let 
\[
	P_{||| \cdot |||}(\tilde{B}^N) = N \( \int_{\{ z \in \re^N : ||| z ||| < 1 \}} dz \)
\]
denote the anisotropic perimeter with respect to the norm $||| \cdot |||$ on $\re^N$.
Then we compute
\begin{align*}
	\int_{\re^N} \frac{dz}{(1 +\sigma ||| z |||^\al)^{\beta}} &= P_{||| \cdot |||}(B^N) \int_0^{\infty} \frac{r^{N-1}}{(1 + \sigma r^\al)^{\beta}} dr \\
	&= N \( \int_{\{ z \in \re^N : ||| z ||| < 1 \}} dz \) \int_0^{\infty} \frac{(\sigma^{-1/\al} s)^{N-1}}{(1 + s^\al)^{\beta}} \sigma^{-1/\al} ds \\
	&= N A_l \sigma^{-\frac{N}{\al}} \int_0^{\infty} \frac{s^{N-1}}{(1 + s^\al)^{\beta}} ds \\
	&= N A_l \sigma^{-\frac{N}{\al}} \frac{\Gamma(\frac{N}{\al}) \Gamma(\beta - \frac{N}{\al})}{\al \Gamma(\beta)}
\end{align*}
which proves lemma.
Here, $A_l$ is defined as in \eqref{Al} in Lemma \ref{Lemma5.1} and we have used a formula
\[
	\int_0^{\infty} \frac{s^c}{(1 + s^a)^b} ds = \frac{\Gamma(\frac{c+1}{\al}) \Gamma(b - \frac{c+1}{a})}{a \Gamma(b)}
\]
for $a,b,c > 0$ such that $ab > c + 1$.
\end{proof}

\vspace{1em}

Recall the (unweighted) sharp $L^p$-logarithmic Sobolev inequality on $\re^n$ with the norm $\| \cdot \|$ for $1 \le p < \infty$:
\begin{equation}
\label{LSp}
	\int_{\re^n} |f|^p \log |f|^p dx \le \frac{n}{p} \log \( \mathcal{L}_p \int_{\re^n} \| \nabla f \|_*^p dx \)
\end{equation}
for any function $f \in W^{1,p}(\re^n) = \{ f \in L^p(\re^n) : \| \nabla f \|_* \in L^p(\re^n) \}$ such that $\int_{\re^n} |f|^p dx = 1$.
Here the sharp constant $\mathcal{L}_p$ is given by
\begin{align*}
	&\mathcal{L}_1 = \frac{1}{n} ( m(B^n) )^{-1/n}, \quad p = 1 \\
	&\mathcal{L}_{p} = \frac{p}{n} \(\frac{p-1}{e}\)^{p-1} \(\Gamma(n/p' + 1) m(B^n) \)^{-p/n}, \quad p > 1.
\end{align*}
where $p' = \frac{p}{p-1}$ for $p > 1$ and $m(B^n) = \int_{B^n} 1 dx$ as before.


The main result in this section is the following:

\begin{theorem}\label{Theorem:equality_p}
Let $\| \cdot \|$ be any norm on $\re^n$ and let $1 < p < \infty$, $p' = \frac{p}{p-1}$.
Then the equality in the sharp $L^p$-logarithmic Sobolev inequality \eqref{LSp},
which holds for any $f \in W^{1,p}(\re^n)$ such that $\int_{\re^n} |f|^p dx = 1$,
occurs if and only if 
\[
	f(x) = \beta \exp \(-\frac{\| x-x_0 \|^{p'}}{\sigma}\)
\]
for some $\sigma>0$, $x_0 \in \re^n$, and $\beta \in \re$ with
\[
	|\beta|^{-p}= \int_{\re^n} \exp \(-\frac{\| x \|^{p'}}{\sigma}\) dx.
\]
\end{theorem}

\begin{proof}
``If" part is checked just by computation. We prove the ``only if" part.

Put $N = nl$ for $l \in \N$.
For a function $f \in C_c^1(\re^n)$ satisfying $\int_{\re^n} |f(x)|^p dx = 1$,
we put
\[
	F(z) = \prod_{i=1}^l f(x^i),
\]
where $x^i \in \re^n$ for each $i = 1,2,\dots,l$,
and $z = (x^1, x^2, \cdots, x^l) \in \re^{nl} = \re^N$.
Define
\[
	||| z ||| = \( \sum_{i=1}^l \| x^i \|^{p'} \)^{1/{p'}}
\]
as before, where $1/p + 1/{p'} = 1$.
Then, we check that the dual norm of $||| \cdot |||$ is given by
\[
	||| \xi |||_{*} = \( \sum_{i=1}^l \| \xi^i \|_*^p \)^{1/p}
\]
for $\xi = (\xi^1, \xi^2, \cdots, \xi^l) \in \re^{nl} = \re^N$,
where $\| \cdot \|_*$ is the dual norm of $\| \cdot \|$ on $\re^n$.
By a direct computation we get
\begin{align}
\label{R1_p}
	&\int_{\re^N} |F(z)|^t dz = \prod_{i=1}^l \int_{\re^n} |f(x^i)|^t dx^i, \quad \forall t \ge 1, \\
\label{R2_p}
	&\int_{\re^N} ||| \nabla F(z) |||_{*}^p dz = l \int_{\re^n} \| \nabla f(x) \|_{*}^p dx.
\end{align}
By checking the proof by V. H. Nguyen \cite{Ng}, which uses the sharp Sobolev inequality with arbitrary norm,
we see that it must hold
\begin{equation}
\label{start_p}
	F_l(z) \sim a_l(1 + b_l ||| z-z_{0,l} |||^{p'})^{1-\frac{N}{p}} 
\end{equation}
for every $z \in \re^N$ with some $a_l>0$, $b_l>0$, and $z_{0,l} \in \re^N$.
Here we have used the notation which emphasizes the dependence on $l$ of involved functions and constants.
This asymptotic formula \eqref{start_p} is our starting point.

By translation invariance, we may fix $z_{0,l} = z_0$ for some fixed point $z_0 \in \re^N$. 
Also recall that we consider functions with $\int_{\re^n} |f(x)|^p dx = 1$, 
we have $\int_{\re^N} |F_l(z)|^p dz = 1$, 
which implies by \eqref{start_p}
\begin{equation}
\label{51}
	a_l^p \int_{\re^N} \frac{dz}{(1 + b_l ||| z - z_0 |||^{p'})^{N-p}} \sim 1.
\end{equation}
Thus by Lemma \ref{Lemma5.2}, Lemma \ref{Lemma5.1} \eqref{Al}, and $N = nl$, we have 
\begin{align*}
	&(a_l)^{1/l} \overset{\eqref{51}}{\sim} \( \int_{\re^N} \frac{dz}{(1 + b_l ||| z - z_0 |||^{p'})^{N-p}} \)^{-\frac{1}{pl}} \\
	&= \( N A_l \frac{\Gamma(\frac{N}{p'}) \Gamma(N-p - \frac{N}{p'})}{p' \Gamma(N-p)} (b_l)^{-\frac{N}{p'}} \)^{-\frac{1}{pl}} \\
	&\overset{\eqref{Al}}{=} \( (nl) \frac{(\frac{n}{p'})^{l-1} (m(B^n))^l (\Gamma(\frac{n}{p'}))^l}{l \Gamma(\frac{nl}{p'})} \frac{\Gamma(\frac{nl}{p'}) \Gamma(nl-p - \frac{nl}{p'})}{p' \Gamma(nl-p)} (b_l)^{-\frac{nl}{p'}} \)^{-\frac{1}{pl}} \\
	&= \underbrace{(\frac{n}{p'})^{-1/p} (m(B^n))^{-1/p} (\Gamma(\frac{n}{p'}))^{-1/p} }_{=C(n,p)} \left\{ \frac{\Gamma(nl -p)}{\Gamma(\frac{nl}{p}-p)} \right\}^{\frac{1}{pl}}  (b_l)^{\frac{n}{pp'}} \\
	&= C(n,p) \left\{ \frac{\Gamma(nl -p)}{\Gamma(\frac{nl}{p}-p)} \right\}^{\frac{1}{pl}}  (b_l)^{\frac{n}{pp'}}.
\end{align*}

Now, Stirling's formula \eqref{Stirling} implies 
\begin{align*}
	&(a_l)^{1/l} \sim C(n,p) \left\{ \frac{(nl-p)^{nl-p-1/2} e^{-(nl-p)}}{(\frac{nl}{p}-p)^{\frac{nl}{p}-p-1/2} e^{-(\frac{nl}{p} -p)}} \right\}^{\frac{1}{pl}} (b_l)^{\frac{n}{pp'}} \\
	&\sim C(n,p) \left\{ \frac{(nl)^{nl-p-1/2}}{(\frac{nl}{p})^{\frac{nl}{p}-p-1/2}} e^{-nl + \frac{nl}{p}} \right\}^{\frac{1}{pl}} (b_l)^{\frac{n}{pp'}} \\
	&= C(n,p) \left\{ p^{\frac{nl}{p}-p-1/2} (nl)^{nl-\frac{nl}{p}} e^{-\frac{nl}{p'}} \right\}^{\frac{1}{pl}} (b_l)^{\frac{n}{pp'}} \\
	&\sim C(n,p) p^{\frac{n}{p^2}} (nl)^{\frac{nl}{p'} \frac{1}{pl}} e^{-\frac{n}{pp'}} (b_l)^{\frac{n}{pp'}} \\
	&= C(n,p) \( \frac{p^{\frac{p'}{p}} (nl)}{e} b_l \)^{\frac{n}{pp'}} 
\end{align*}
as $l \to \infty$. 
Therefore, we obtain
\begin{equation}
\label{all}
	(a_l)^{1/l} \sim C(n,p) \( \frac{p^{\frac{p'}{p}} n}{e} \)^{\frac{n}{pp'}} (l b_l)^{\frac{n}{pp'}} \quad \text{as $l \to \infty$}.
\end{equation}
Recall $F_l(z) = \Pi_{i=1}^l f_l(x^i)$ where $z = (x^1, \cdots, x^l) \in \re^{nl} = \re^N$, $x^i \in \re^n$ $(i=1,\cdots, l)$.
We choose
$$
	z=(x,\cdots,x), \quad z_0=(x_0,\cdots,x_0) \quad \text{with } x,x_0 \in \re^n.
$$
Then it follows that 
\[
	||| z - z_0 |||^{p'} =  \sum_{i=1}^l \| x^i - x_0^i \|^{p'}= l \| x - x_0 \|^{p'}
\]
and
$$
	f_l(x) = (F_l(z))^{1/l} \sim (a_l)^{1/l} (1 + l b_l \| x-x_0 \|^{p'})^{1/l-n/p} \quad (\forall x \in \re^n)
$$
as $l \to \infty$ by \eqref{start_p}.
Note that $(f_l(x_0))^l = F_l(z_0) \sim a_l$ as $l \to +\infty$.

We have three possible behaviors of the sequence $b_l$ as $l \to \infty$:
\begin{itemize}
\item [i)] $b_l\rightarrow+\infty$,
\item [ii)] $b_l\rightarrow \bar{b}\in (0, +\infty)$,
\item [iii)]$b_l\rightarrow 0$.
\end{itemize}
Indeed if the limit does not exist, then we can argue one of the previous cases up to a subsequence.

If $b_l \to +\infty$, we have the contradiction since by \eqref{all} 
\begin{align*}
	f_l(x) &\sim (a_l)^{1/l} (1 + l b_l \| x-x_0 \|^q)^{1/l-n/p} \\
	&\sim C(n,p) \( \frac{p^{\frac{p'}{p}} n}{e} \)^{\frac{n}{pp'}} (l b_l)^{\frac{n}{pp'}} (1 + l b_l \| x-x_0 \|^{p'})^{1/l-n/p} \\
	&\sim C(n,p) \( \frac{p^{\frac{p'}{p}} n}{e} \)^{\frac{n}{pp'}} (l b_l)^{\frac{n}{pp'} + \frac{1}{l} - \frac{n}{p}} \to 0
\end{align*}
as $l \to +\infty$ for any $x \in \re^n$ with $x \ne x_0$,
which is absurd by the restriction $\int_{\re^n} |f_l(x)|^p dx = 1$.
Note that $\frac{n}{pp'} + \frac{1}{l} - \frac{n}{p} < 0$ when $l$ is sufficiently large.

Also if $b_l \to \bar{b} \in (0, +\infty)$, we see
\begin{align*}
	f_l(x) &\sim (a_l)^{1/l} (1 + l b_l \| x-x_0 \|^q)^{1/l-n/p} \\
	&\sim C(n,p) \( \frac{p^{\frac{p'}{p}} n}{e} \)^{\frac{n}{pp'}} (\bar{b})^{\frac{n}{pp'} - \frac{n}{p}} l^{\frac{n}{pp'} + \frac{1}{l} - \frac{n}{p}} \to 0
\end{align*}
as $l \to +\infty$ by \eqref{all} for any $x \in \re^n$, $x \ne x_0$, again a contradiction.

Thus the only possible case is the third one and we have $b_l \to 0$ as $l \to \infty$.

Next, we choose
$$
	z=(x,\underbrace{x_0 \cdots,x_0}_{l-1}), \quad z_0=(x_0,\cdots,x_0) \quad \text{with} x,x_0 \in \re^n.
$$
Then it follows that $||| z - z_0 |||^{p'}  = \sum_{i=1}^l \| x^i - x_0^i\|^{p'} = \| x - x_0 \|^{p'}$ and since $F_l(z) = f_l(x) (f_l(x_0))^{l-1} \sim (a_l)^{\frac{l-1}{l}} f_l(x)$, 
we have by \eqref{start_p} that
$$
	f_l(x) \sim (a_l)^{1/l} (1 + b_l \| x-x_0 \|^{p'})^{1-N/p}, \quad \forall x \in \re^n
$$
as $l \to \infty$.
In this case, by $\log(1 + t) = t + o(1)$ as $t \to +0$, we see
\begin{align*}
	f_l(x) &\sim (a_l)^{1/l} (1 + b_l \| x-x_0 \|^{p'})^{1-N/p} \\
	&= (a_l)^{1/l} \exp \( (1 - N/p) \log (1 +  b_l \| x-x_0 \|^{p'}) \) \\
	&\sim (a_l)^{1/l} \exp \( (1 - N/p)   b_l \| x-x_0 \|^{p'}) \).
\end{align*}

Again three cases (up to a subsequence) are possible for the behaviors of the sequence $lb_l$,
\begin{enumerate}
\item[ \textbf{(iii-1)}] $lb_l\rightarrow+\infty$,
\item[ \textbf{(iii-2)}] $lb_l\rightarrow 0$,
\item[ \textbf{(iii-3)}] $lb_l\rightarrow \bar{b}\in (0, +\infty)$.
\end{enumerate}

The first case (iii-1) does not occur since by \eqref{all}, we have
\begin{align*}
	f_l(x) &\sim (a_l)^{1/l} \exp \( (1 - \frac{nl}{p})   b_l \| |x-x_0 \|^{p'}) \) \\
	&\sim C(n,p) \( \frac{p^{\frac{p'}{p}} n}{e} \)^{\frac{n}{pp'}} (l b_l)^{\frac{n}{pp'}} \exp \( (- \frac{n}{p})  (l b_l) \| |x-x_0 \|^{p'}) \) \to 0
\end{align*}
as $l \to +\infty$ for any $x \in \re^n$ such that $x \ne x_0$, if $l b_l \to +\infty$.

Also the second case (iii-2) does not occur, since if it would happen, then
\[
	a_l^{1/l} \sim C(n,p) \( \frac{p^{\frac{{p'}}{p}} n}{e} \)^{\frac{n}{p{p'}}} (l b_l)^{\frac{n}{p{p'}}} \to 0
\]
and $b_l \to 0$ as $l \to \infty$. 
Thus
\begin{align*}
	f_l(x) \sim (a_l)^{1/l} (1 + b_l \| x-x_0 \|^{p'})^{1-(nl)/p} \to 0
\end{align*}
as $l \to +\infty$ for any $x \in \re^n$, which is absurd.

Thus the only case to be considered is the third one: $lb_l\rightarrow \bar{b} \in (0, +\infty)$.
In this case, we have
\begin{align*}
	f_l(x) &\sim (a_l)^{1/l} \exp \( (1 - N/p)   b_l \| x-x_0 \|^{p'}) \) \\
	&\sim C(n,p) \( \frac{p^{\frac{p'}{p}} n}{e} \)^{\frac{n}{pp'}} (l b_l)^{\frac{n}{pp'}} \exp \( (- \frac{n}{p})  (l b_l) \| |x-x_0 \|^{p'}) \) \\
	&\sim C(n,p) \( \frac{p^{\frac{p'}{p}} n}{e} \)^{\frac{n}{pp'}} (\bar{b})^{\frac{n}{pp'}} \exp \( (- \frac{n}{p})  \bar{b} \| |x-x_0 \|^{p'}) \)
\end{align*}
as $l \to +\infty$. 
This is the desired conclusion.
\end{proof}

%
%

\vspace{1em}\noindent
{\bf Acknowledgments.}

This work was partly supported by Osaka City University Advanced Mathematical Institute MEXT Joint Usage / Research Center on Mathematics and Theoretical Physics JPMXP0619217849.
The first author (F.F.) has been partially supported by GNAMPA - INdAM.
The second author (F.T.) was supported by JSPS KAKENHI Grant-in-Aid for Scientific Research (B), JP19H01800, and JSPS Grant-in-Aid for Scientific Research (S), JP19H05597.

\end{document}